\documentclass[12pt,a4paper]{article}

\unitlength\textwidth
\divide\unitlength by 300\relax

\usepackage{amsmath,amssymb}
\usepackage{bm}
\bmdefine{\sss}{s}

\newcommand{\NNN}{\mathbb{N}}

\newcommand{\ZZZ}{\mathbb{Z}}

\newcommand{\RRRRR}{{\mathcal R}}
\newcommand{\TTTTT}{{\mathcal T}}

\newcommand{\mmmm}{{\mathfrak{m}}}

\newcommand{\covers}{\mathrel{\cdot\!\!\!>}}
\newcommand{\covered}{\mathrel{<\!\!\!\cdot}}

\newcommand{\define}{\mathrel{:=}}

\newcommand{\gor}{Gorenstein}
\newcommand{\agor}{almost Gorenstein}
\newcommand{\Agor}{Almost Gorenstein}
\newcommand{\cm}{Cohen-Macaulay}
\newcommand{\noeth}{Noetherian}

\newcommand{\qed}{\nolinebreak\rule{.3em}{.6em}}
\newcommand{\joinirred}{join-irreducible}
\newcommand{\rank}{\mathrm{rank}}

\newcommand{\join}{\vee}

\newcommand{\image}{{\mathrm{Im}}}
\newcommand{\cok}{{\mathrm{Cok}}}

\newcommand{\starcpx}{{\mathrm{star}}}

\newcommand{\imply}{\Longrightarrow}
\newcommand{\condn}{{condition N}}
\newcommand{\scn}{{sequence with condition N}}

\newcommand{\rmax}{{r_{\max}}}
\newcommand{\irseq}{{irredundant sequence}}

\newcommand{\mylabel}{\label}

\newtheorem{thm}{Theorem}[section]
\newtheorem{fact}[thm]{Fact}

\newtheorem{lemma}[thm]{Lemma}
\newtheorem{cor}[thm]{Corollary}
\newtheorem{definition}[thm]{Definition}

\newtheorem{remark}[thm]{Remark}

\numberwithin{equation}{section}

\newcounter{cond}
\numberwithin{cond}{section}

\makeatletter
\newcommand{\mysloppy}{\tolerance 9999 \hfuzz .5\p@ \vfuzz .5\p@}
\makeatother
\title{%
Almost \gor\ Hibi rings}
\author{Mitsuhiro MIYAZAKI\footnote{%
The author is supported partially by 
JSPS KAKENHI Grant Number 15K04818.}%
}
\date{%
Dept.\ Math.,
Kyoto University of Education,
\\
Fushimi-ku, Kyoto, 612-8522, Japan}

\begin{document}
\sloppy

\maketitle

\begin{abstract}
In this paper, we state criteria of a Hibi ring to be 
level, non-\gor\ and \agor\ and to be non-level and \agor\
in terms of the structure of the partially ordered set defining the Hibi ring.
We also state a criterion of a ladder determinantal ring defined by 2-minors
to be non-\gor\ and \agor\ in terms of the shape of the ladder.
\\
Keywords:
\agor, Hibi ring, level ring, ladder determinantal ring
\\
MSC:13H10, 13F50, 13A02
\end{abstract}

\section{Introduction}

\cm\ and \gor\ properties are traditional and very important notions in 
commutative ring theory.
First, \cm\ property of a \noeth\ ring is defined as a ring which satisfies unmixedness theorem
and/or other equivalent conditions.
After that, Bass \cite{bas} defined \gor\ property for \noeth\ rings.
A \noeth\ local ring is by definition \gor\ if its self injective dimension is finite and
a general \noeth\ ring is \gor\ if the localized ring by any prime ideal of it is a 
\gor\ local ring.

Since then, researchers have accumulated the theories concerning \cm\ and \gor\ rings and had
feeling that in some theories, assuming only \cm\ property is to weak to deduce interesting results
while \gor\ property is to strong so that only trivial results can be deduced.
Therefore, there are efforts to find and formulate a 
nice class of rings to fill the gap between \cm\ and
\gor\ properties.

The notion of \agor\ property is an outcome of these efforts.
First Barucci and Fr\"oberg \cite{bf} defined \agor\ property for 1-dimensional analytically 
unramified local rings.
After that, Goto, Matsuoka and Phuong
\cite{gmp} generalized the notion of \agor\ property to arbitrary 1-dimensional 
local rings.
Further, Goto, Takahashi and Taniguchi
\cite{gtt} generalized the notion of \agor\ property to arbitrary local rings 
and graded rings.

In this paper, we study \agor\ property of Hibi rings and completely
characterize non-\gor\ \agor\ graded Hibi rings in terms of the structure
of partially ordered sets.
Let $k$ be a field, $H$ a finite distributive lattice, $P$ the set of 
\joinirred\ elements of $H$ and $\RRRRR_k(H)$ the Hibi ring over 
$k$ on $H$.
Then $\RRRRR_k(H)$ is non-\gor\ and \agor\ if and only if,
after deleting elements $x\in P$ with $\starcpx_P(x)=P$, $P$ is one of the 
following forms.
\begin{enumerate}
\item
\label{item:int level agor}
Disjoint union of an element and a chain with length at least 1.
\item
\label{item:int non level agor}
A partially ordered set
 obtained by adding one of the following covering relations to disjoint union
of 2 chains with the same length of at least 2.
\begin{enumerate}
\item
\label{item:int non level sym}
The bottom element of each chain is covered by the top element of the other chain.
\item
\label{item:int non level pure}
The bottom element of the second chain is covered by the top element of the first chain
and there are integers $i_1$, \ldots, $i_p$ such that the $i_j$-th element of the first chain from
the bottom is covered by the $(i_j+1)$-th element of the second chain from the bottom for 
each $j=1$, \ldots, $p$
($p$ may be 0 in this case).
\end{enumerate}
\end{enumerate}
$$
\vcenter{%\hsize=\textwidth\relax
\begin{picture}(50,160)
\put(10,90){\circle*{4}}

\put(40,30){\circle*{4}}
\put(40,50){\circle*{4}}
\put(40,70){\circle*{4}}

\put(40,110){\circle*{4}}
\put(40,130){\circle*{4}}
\put(40,150){\circle*{4}}

\put(40,30){\line(0,1){50}}
\put(40,100){\line(0,1){50}}

\multiput(40,82)(0,2){9}{\makebox(0,0){$\cdot$}}

\put(25,10){\makebox(0,0){\ref{item:int level agor}}}

\end{picture}
\hfill
\begin{picture}(50,160)
\put(10,30){\circle*{4}}
\put(10,50){\circle*{4}}
\put(10,70){\circle*{4}}

\put(10,110){\circle*{4}}
\put(10,130){\circle*{4}}
\put(10,150){\circle*{4}}

\put(40,30){\circle*{4}}
\put(40,50){\circle*{4}}
\put(40,70){\circle*{4}}

\put(40,110){\circle*{4}}
\put(40,130){\circle*{4}}
\put(40,150){\circle*{4}}

\put(10,30){\line(0,1){50}}
\put(10,100){\line(0,1){50}}

\put(40,30){\line(0,1){50}}
\put(40,100){\line(0,1){50}}

\put(10,30){\line(1,4){30}}
\put(10,150){\line(1,-4){30}}

\multiput(10,82)(0,2){9}{\makebox(0,0){$\cdot$}}
\multiput(40,82)(0,2){9}{\makebox(0,0){$\cdot$}}

\put(25,10){\makebox(0,0){\ref{item:int non level sym}}}

\end{picture}
\hfill
%\vcenter{\parindent0pt\hsize=40\unitlength
\begin{picture}(50,160)
\put(10,30){\circle*{4}}
\put(10,50){\circle*{4}}
\put(10,70){\circle*{4}}

\put(10,110){\circle*{4}}
\put(10,130){\circle*{4}}
\put(10,150){\circle*{4}}

\put(40,30){\circle*{4}}
\put(40,50){\circle*{4}}
\put(40,70){\circle*{4}}

\put(40,110){\circle*{4}}
\put(40,130){\circle*{4}}
\put(40,150){\circle*{4}}

\put(10,30){\line(0,1){50}}
\put(10,100){\line(0,1){50}}

\put(40,30){\line(0,1){50}}
\put(40,100){\line(0,1){50}}

\put(10,30){\line(3,2){30}}
\put(10,110){\line(3,2){30}}
\put(10,150){\line(1,-4){30}}

\multiput(10,82)(0,2){9}{\makebox(0,0){$\cdot$}}
\multiput(40,82)(0,2){9}{\makebox(0,0){$\cdot$}}

\put(25,10){\makebox(0,0){\ref{item:int non level pure}}}

\end{picture}
}
$$
See Theorems \ref{thm:level main} and \ref{thm:non level main}.

We also state a criterion of non-\gor\ \agor\ property of ladder
determinantal rings defined by 2-minors, since (ladder) determinantal
rings defined by 2-minors have structures of Hibi rings.
See Theorem \ref{thm:ladder}.

This paper is organized as follows.
After establishing basic materials used in the main argument of this paper
in Section \ref{sec:prel},
we state a criterion of a Hibi ring to be level, non-\gor\ and \agor\ graded ring in 
Section \ref{sec:level}.
Further, we state a criterion of a Hibi ring to be non-level and \agor\ graded ring in
Section \ref{sec:non level}.
We show that \ref{item:int level agor} above corresponds to the level case and
\ref{item:int non level agor} to the non-level case.
Finally in Section \ref{sec:ladder}, we state a criterion of a ladder determinantal
ring defined by 2-minors to be non-\gor\ and \agor.

\section{Preliminaries}
\label{sec:prel}

In this section, we collect some known facts and state basic facts in order to
prepare the main argument of this paper.

In this paper, all rings and algebras are assumed to be commutative 
with an identity element.
We denote by $\NNN$ the set of non-negative integers,
by $\ZZZ$ the set of integers.
%, by $\RRR$ the set of real numbers and by $\RRR_{\geq0}$ the set of non-negative real numbers.
Let $k$ be a field, $R$ a \cm\ standard graded $k$-algebra, i.e.,
$R=\bigoplus_{n\geq0}R_n$, $R_0=k$ and $R=k[R_1]$.
We set $d=\dim R$, $a=a(R)$, where $a(R)$ is the $a$-invariant of $R$.
We denote the irrelevant maximal ideal $\bigoplus_{n>0}R_n$ of $R$ by $\mmmm$,
the graded canonical module of $R$ by $K_R$.

For a graded finitely generated $R$-module $M$, we denote the Hilbert series
$\sum_{i}(\dim_k M_i)\lambda^i$ by $[[M]]$.
We set $(h_0,h_1,\ldots, h_s)$, $h_s\neq 0$, the $h$-vector of $R$, i.e.,
$$
[[R]]=\frac{h_0+h_1\lambda+\cdots+h_s\lambda^s}{(1-\lambda)^d}.
$$
It is known that $s=a+d$.
We also denote the number of elements in a minimal system of generators of $M$
by $\mu(M)$ and the multiplicity of $M$ with respect to $\mmmm$ by $e(M)$.
It is known that $\mu(M)\leq e(M)$.
Further the following fact is known.

\begin{lemma}
\label{lem:mum em}
Let $M$ be a \cm\ non-negatively graded $R$-module with $\dim M=u$.
Then there is a polynomial $F(\lambda)\in\ZZZ[\lambda]$ such that
$$
[[M]]=\frac{F(\lambda)}{(1-\lambda)^u}.
$$
Further if we set $F(\lambda)=\sum_{i}c_i\lambda^i$, then
$$
\dim_k [M/\mmmm M]_i\leq c_i
$$
for any $i$.
Moreover, if $\mu(M)=e(M)$, then the equality holds for any $i$.
\end{lemma}

Now we recall the definition of an \agor\ ring in our case.

\begin{definition}[{\cite{gtt}}]
\rm
$R$ is an \agor\ graded ring if there exist an exact sequence
$$
0\to R\to K_R(-a)\to C\to 0
$$
of graded $R$-modules such that $C=0$ or $\mu(C)=e(C)$.
(We always assume a homomorphism between graded $R$-modules is of
degree 0.)
\end{definition}
In \cite{gtt}, notions of \agor\ local rings and \agor\ graded rings
are defined.
In this paper, we treat only \agor\ graded rings and just call them \agor\ rings.

Note that if
$$
0\to R\stackrel{\varphi}{\to} K_R(-a)\stackrel{\pi}{\to} C\to 0
$$
is an exact sequence, then
$\varphi(1)$ is a generator of $K_R(-a)$.
Set $\xi_1=\varphi(1)$ and take a minimal homogeneous generating system
$\xi_1$, $\xi_2$, \ldots, $\xi_p$ of $K_R(-a)$.
Then $\pi(\xi_2)$, \ldots, $\pi(\xi_p)$ is a minimal homogeneous
generating system of $C$.
In particular, we see the following fact.

\begin{lemma}
\label{lem:dim k c}
In the above situation, 
$$
\dim_k [K_R(-a)/\mmmm K_R(-a)]_i=\dim_k [C/\mmmm C]_i+\delta_{0,i},
$$
where $\delta_{0,i}$ is the Kronecker's delta.
\end{lemma}
Note that if $C\neq 0$, then $C$ is a \cm\ $R$-module of dimension $d-1$
\cite[Lemma 3.1]{gtt}.
We also recall the following fact.

\begin{fact}[{\cite[Proposition 2.4]{hig}}]
\label{fac:hilb c}
If $C\neq 0$, then
$$
[[C]]=\frac{\sum_{j=0}^{s-1}((h_s+\cdots+h_{s-j})-(h_0+\cdots+h_j))\lambda^j}
{(1-\lambda)^{d-1}}.
$$
\end{fact}
By Lemmas \ref{lem:mum em} and \ref{lem:dim k c} and Fact \ref{fac:hilb c},
we see the following

\begin{cor}
\label{cor:sym dim}
Suppose that $C\neq 0$ and $\mu(C)=e(C)$. Then
$$\dim_k[C/\mmmm C]_i=\dim_k[C/\mmmm C]_{s-1-i}$$
 for $0\leq i\leq s-1$,
$$\dim_k[K_R(-a)/\mmmm K_R(-a)]_i=\dim_k[K_R(-a)/\mmmm K_R(-a)]_{s-1-i}$$
 for $1\leq i\leq s-2$
and
$$\dim_k[K_R(-a)/\mmmm K_R(-a)]_{s-1}=\dim_k[K_R(-a)/\mmmm K_R(-a)]_{0}-1$$
 if $s\geq 2$.
\end{cor}

Next we recall some basic facts about Hibi rings.
First we fix notation and terminology on partially ordered set 
(poset for short).

Let $Q$ be a finite poset.
A chain $X$ in $Q$ is a totally ordered subset of $Q$.
The length of a chain $X$ is by definition $\#X-1$,
where $\#X$ is the cardinality of $X$.
The rank of $Q$, denoted $\rank Q$, is the maximum length of chains in $Q$.
If every maximal chain (with respect to the inclusion relation) has
length $\rank Q$, we say that $Q$ is pure.
A subset $I$ of $Q$ such that $x\in Q$, $y\in I$ and $x<y$ imply
$x\in I$ is called a poset ideal of $Q$.
For $x$, $y\in Q$ with $x\leq y$ we define
$[x,y]_Q\define \{z\in Q\mid x\leq z\leq y\}$.
We denote $[x,y]_Q$ as $[x,y]$ if there is no danger of confusion.
$(x,y]_Q$, $[x,y)_Q$ and $(x,y)_Q$ for $x$, $y\in Q$ with $x<y$ are
defined similarly.
If $x$, $y\in Q$, $x<y$ and $(x,y)_Q=\emptyset$,
we say
that $y$ covers $x$ and denote $x\covered y$ or $y\covers x$.
For $x\in Q$, we define $\starcpx_Q(x)\define\{y\in Q\mid y\leq x$
or $y\geq x\}$.
For a poset $Q$, let $\infty$ be a new element which is not contained in
$Q$.
We define the poset $Q^+$ whose base set is $Q\cup\{\infty\}$ and 
$x<y$ if and only if $x$, $y\in Q$ and $x<y$ in $Q$ or
$x\in Q$ and $y=\infty$.

Let $H$ be a finite distributive lattice with unique minimal element
$\{x_0\}$, $P$ the set of \joinirred\ elements of $H$, i.e.,
$P=\{\alpha \in H\mid \alpha=\beta\join\gamma\Rightarrow
\alpha=\beta$ or $\alpha=\gamma\}$.
Note that we treat $x_0$ as a \joinirred\ element.
It is known that $H$ is isomorphic to the set of non-empty poset ideals of $P$
ordered by inclusion.

Let $\{T_x\}_{x\in P}$ be a family of indeterminates indexed by $P$.

\begin{definition}[{\cite{hib}}]
\rm
$\RRRRR_k(H)\define k[\prod_{x\leq \alpha}T_x\mid \alpha\in H]$.
\end{definition}
$\RRRRR_k(H)$ is called the Hibi ring over $k$ on $H$ nowadays.
Hibi \cite[\S 2 b)]{hib} showed that $\RRRRR_k(H)$ is a normal affine 
semigroup ring and thus is \cm\ by the result of Hochster \cite{hoc}.
Further, by setting $\deg T_{x_0}=1$ and $\deg T_x=0$ for $x\in P\setminus\{x_0\}$,
$\RRRRR_k(H)$ is a standard graded $k$-algebra.
From now on, we set $R=\RRRRR_k(H)$.

Here we recall the characterization of \gor\ property of $R$ by Hibi.

\begin{fact}[{\cite[\S 3 d)]{hib}}]
\label{fac:gor cri}
$R$ is \gor\ if and only if $P$ is pure.
\end{fact}

Set
$\overline\TTTTT(P)\define\{\nu\colon P^+\to\NNN\mid \nu(\infty)=0$
and $x< y\imply \nu(x)\geq\nu(y)\}$
and
$\TTTTT(P)\define\{\nu\colon P^+\to\NNN\mid \nu(\infty)=0$
and $x<y\imply \nu(x)>\nu(y)\}$.
Further, we set $T^{\nu}\define\prod_{x\in P}T_x^{\nu(x)}$ for $\nu\in\overline\TTTTT(P)$.
Note that $\deg T^\nu=\nu(x_0)$.
It is easily verified that
$$
R=\bigoplus_{\nu\in\overline\TTTTT(P)}kT^\nu
$$
\cite[\S3 b)]{hib}
and therefore, by the description of the canonical module of a normal affine 
semigroup ring by Stanley \cite[p.\ 82]{sta2}, we see that
$$
K_R=\bigoplus_{\nu\in\TTTTT(P)}kT^\nu
$$
\cite[\S3 b)]{hib}.
We define the order on $\TTTTT(P)$ by
$\nu\leq\nu'\iff \nu'-\nu\in\overline\TTTTT(P)$,
where we set $(\nu'-\nu)(x)\define\nu'(x)-\nu(x)$.
Then for $\nu\in\TTTTT(P)$, $T^\nu$ is a generator of 
$K_R$ if and only if $\nu$ is a minimal element of $\TTTTT(P)$.

Since $\nu(x)\geq\rank[x,\infty]$ for any $x\in P$ and $\nu\in\TTTTT(P)$,
and $\nu_0\colon P^+\to\NNN$, $x\mapsto \rank[x,\infty]$
is an element of $\TTTTT(P)$, we see that
$$
\min\{i\mid[K_R]_i\neq0\}=\rank P^+.
$$
In particular, $\rank P^+=-a$.
In order to use notation of \cite{miy}, we define auxiliary notation
by $r\define\rank P^+$.
We also see by \cite[\S2 a)]{hib} that $d=\#P$.

We recall the notion of a \scn\  defined in \cite{miy}.

\begin{definition}\rm
\mylabel{def:condn}
Let $y_1$, $x_1$, $y_2$, $x_2$, \ldots, $y_t$, $x_t$ be a 
(possibly empty) sequence of elements in $P$.
We say the sequence 
$y_1$, $x_1$, $y_2$, $x_2$, \ldots, $y_t$, $x_t$ satisfies \condn\
if 
\begin{enumerate}
\item
\mylabel{item:cond n 0}
$x_1\neq x_0$.
\item\mylabel{item:cond n 1}
$y_1>x_1<y_2>x_2<\cdots<y_t>x_t$.
\item\mylabel{item:cond n 2}
For any $i$, $j$ with $1\leq i<j\leq t$,
$y_i\not\geq x_j$.
\end{enumerate}
\end{definition}

\begin{remark}
\rm
A \scn\ may be an empty sequence, i.e.,  $t$ may be $0$.
\end{remark}
For a \scn, we make the following

\begin{definition}\rm
Let
$y_1$, $x_1$, $y_2$, $x_2$, \ldots, $y_t$, $x_t$ 
be a \scn.
We set
$$
r(y_1, x_1, \ldots, y_t, x_t)\define
\sum_{i=1}^t(\rank[x_{i-1},y_i]-\rank[x_i,y_i])+\rank[x_t,\infty],
$$
where we set an empty sum to be 0.
\end{definition}

\begin{remark}\rm
For an empty sequence, we set $r()=\rank[x_0,\infty]=r$.
\end{remark}

It is fairly easy to see that there are only  finitely many sequences with \condn\
\cite[Lemma 3.4]{miy} and we state the following

\begin{definition}\rm
We set
$\rmax\define\max
\{r(y_1,x_1,\ldots,y_t,x_t)\mid y_1$, $x_1$, \ldots, $y_t$, $x_t$ 
is a \scn$\}$.
\end{definition}
By \cite[Corollary 3.5 and Theorem 3.12]{miy}, we see the following 

\begin{fact}
\label{fac:can gen deg}
$
[K_R/\mmmm K_R]_i\neq 0\iff r\leq i\leq \rmax
$.
\end{fact}

\section{A criterion of a Hibi ring to be level, non-\gor\ and \agor}

\label{sec:level}

In this section and the next, we state criteria of a Hibi ring to be non-\gor\
and \agor.
Recall that we have set $R=\RRRRR_k(H)$, $P$ the set of \joinirred\ elements of $H$,
$r=\rank P^+$, $d=\dim R$ and $a=a(R)$.
Also recall that $(h_0,\ldots, h_s)$, $h_s\neq0$, is the $h$-vector of $R$.

In this section, we treat the level case.

\begin{thm}
\label{thm:level main}
$R$ is level, non-\gor\ and \agor\ if and only if there exists $z\in P$
with the following conditions.
\begin{enumerate}
\item
\label{item:chain r}
$P^+\setminus\{z\}$ is a chain of length $r$.
\item
\label{item:non pure}
$\rank[x_0,z]+\rank[z,\infty]<r$.
\end{enumerate}
\end{thm}
\begin{proof}
We first prove the ``only if'' part.
Since $R$ is not \gor, it is not regular.
Therefore $s\geq 1$.
Take an exact sequence
$$
0\to R\to K_R(-a)\to C\to 0
$$
such that $\mu(C)=e(C)$ (since $R$ is not \gor, $C\neq 0$).

Since $R$ is level, we see that 
$[K_R(-a)/\mmmm K_R(-a)]_i=0$ for $i>0$.
Therefore, we see by Lemma \ref{lem:dim k c} that
$[C/\mmmm C]_i=0$ for $1\leq i\leq s-1$.
If $s\geq 2$, then we see by Corollary \ref{cor:sym dim} that
$\dim_k[C/\mmmm C]_0=\dim_k[C/\mmmm C]_{s-1}=0$.
This contradicts to $C\neq 0$.
Therefore, we see that $s=1$.

Take a maximal chain 
$$
x_0<x_1<\cdots<x_{r-1}<x_r=\infty
$$
of length $r$ in $P^+$.
Since $1=s=a+d=d-r=\#P-\rank P^+$,
we see that 
$$
\#(P^+\setminus\{x_0,x_1,\ldots, x_r\})=1.
$$
Set $\{z\}=P^+\setminus\{x_0,x_1,\ldots, x_r\}$.
We show that $z$ satisfies \ref{item:chain r} and \ref{item:non pure}.
\ref{item:chain r} is clear.
Further, $\rank[x_0,z]+\rank[z,\infty]<r$ since $P^+$ is not pure
by Fact \ref{fac:gor cri}.
%%%%%%%%%%%%%%%%%%%%%%
\iffalse
Set $u=\max\{i\mid x_i<z\}$ and $v=\min\{i\mid x_i>z\}$.
Then $v-u\geq 3$ since $P^+$ is not pure by Fact \ref{fac:gor cri}.
Therefore,
$$
\rank[x_0,z]+\rank[z,\infty]=(u+1)+(r-v+1)\leq r-1.
$$
\fi
%%%%%%%%%%%%%%%%%%%%%%%%%%%%%%%%%%%%
Thus, $z$ satisfies \ref{item:non pure}.

Next we prove the ``if'' part.
By \ref{item:chain r} and \ref{item:non pure}, we see that $P^+$ is not pure
and
$$
s=a+d=d-r=\#P-\rank P^+=1.
$$
Therefore, $R$ is not \gor, level and \agor\ by Fact \ref{fac:gor cri}
and \cite[Proposition 10.1 and Theorem 10.4]{gtt}.
(Note that in \cite[\S10]{gtt}, they assumed that the base field is an infinite
field.
However, \cite[Lemma 10.1 and Theorem 10.4]{gtt} also hold true in our situation.
First, if $s=1$, then
$R$ has a 2-linear resolution as a module over a
polynomial ring over $k$ with $\dim_k R_1$ variables.
Therefore, $R$ is level.
Thus, \cite[Lemma 10.1]{gtt} also valid in our situation.
As for \cite[Lemma 10.3]{gtt}, which is used in the proof of 
\cite[Theorem 10.4]{gtt} is clearly valid in our situation, since $R$ is a domain
and $K_R$ is an ideal of $R$.
Further, 
we can use \cite[Theorem 10.4]{gtt} in our case by the argument of
base field extension, since $R\otimes_k k'$ is a domain for any extension field
$k'$ of $k$.
Level property of $R$ may also be deduced from \cite[Theorem 3.9]{miy}
and \agor\ property of $R$ is also proved directly by the same argument
as in the proof of Theorem \ref{thm:non level main}.)
\end{proof}

\begin{remark}\rm
If $R$ satisfies the conditions in Theorem \ref{thm:level main},
then $R$ is a polynomial ring over the ring considered in \cite[Example 10.5]{gtt}.
\end{remark}

\section{A criterion of a Hibi ring to be non-level and \agor}

\label{sec:non level}

In this section, we state a criterion of a Hibi ring to be
non-level and \agor.

\begin{thm}
\label{thm:non level main}
$R$ is not level and \agor\ if and only if there are non-negative 
integers $m$, $n$ and $n'$ with $m\geq 3$,
$z_1$, \ldots, $z_m$, $z'_1$, \ldots, $z'_m$, $w_0$, \ldots, $w_n$,
$w'_0$, \ldots, $w'_{n'}\in P^+$ with the following conditions.
\begin{enumerate}
\item
\label{item:whole set}
$P^+=\{z_1$, \ldots, $z_m$, $z'_1$, \ldots, $z'_m$, 
$w_0$, \ldots, $w_n$, $w'_0$, \ldots, $w'_{n'}\}$.
\item
\label{item:least cover rel}
\begin{description}
\item
$x_0=w_0\covered w_1\covered\cdots\covered w_n$,
\item
$w'_0\covered w'_1\covered\cdots\covered w'_{n'}=\infty$,
\item
$w_n\covered z_1\covered z_2\covered \cdots\covered z_m\covered w'_0$,
\item
$w_n\covered z'_1\covered z'_2\covered \cdots\covered z'_m\covered w'_0$ and
\item
$z'_1\covered z_m$.
\end{description}
\item
\label{item:other cover rel}
Other covering relations in $P^+$ are one of the followings.
\begin{enumerate}
\item
\label{item:sym cover}
$z_1\covered z'_m$ or
\item
\label{item:pure cover}
there are integers $i_1$, \ldots, $i_p$ with
$1\leq i_1<\cdots<i_p\leq m-1$ such that
$$
z_{i_j}\covered z'_{i_j+1}\quad\mbox{for $j=1$, \ldots, $p$.}
$$
($p$ may be $0$ in this case, i.e., there is no covering relation other than 
stated in 
\ref{item:least cover rel}.)
\end{enumerate}
\end{enumerate}
\end{thm}
In order to prove this theorem, we first show the following fact.
\begin{lemma}
\label{lem:no retrace}
Let $\nu$ be a minimal element of $\TTTTT(P)$.
Then there exists a sequence $y_1$, $x_1$, \ldots, $y_t$, $x_t$ of elements in $P$
with the following conditions, where we set $y_{t+1}=\infty$.
\begin{enumerate}
\item
\label{item:condn}
$y_1$, $x_1$, \ldots, $y_t$, $x_t$ satisfies \condn.
\item
\label{item:nu diff rank}
$\nu(x_i)-\nu(y_{i+1})=\rank[x_i,y_{i+1}]$ for $0\leq i\leq t$.
\item
\label{item:no retrace}
For any $i$ with $1\leq i\leq t$,
$\rank[x_i,z]+\rank[z,y_{i+1}]
<\rank[x_i,y_{i+1}]$
for any $z\in (x_i,y_{i}]\cap (x_i,y_{i+1}]$
and
$\rank[x_{i-1},z]+\rank[z,y_i]
<\rank[x_{i-1},y_{i}]$
for any $z\in [x_{i-1},y_{i})\cap [x_i,y_i)$.
\end{enumerate}
\end{lemma}
\begin{proof}
By \cite[Lemma 3.3]{miy}, we see that there is a sequence
$y_1$, $x_1$, \ldots, $y_t$, $x_t$ of elements in $P$ with \condn\
such that
$\nu(x_i)-\nu(y_{i+1})=\rank[x_i,y_{i+1}]$ for $0\leq i\leq t$.

If there exist $i$ with $1\leq i\leq t$ and 
$z\in (x_i,y_{i}]\cap (x_i,y_{i+1}]$
such that 
$$
\rank[x_i,z]+\rank[z,y_{i+1}]=\rank[x_i,y_{i+1}],
$$
then we can replace $x_i$ with $z$.
Similarly, if there exist $i$ with $1\leq i\leq t$ and 
$z\in [x_{i-1},y_{i})\cap [x_i,y_i)$ such that
$$
\rank[x_{i-1},z]+\rank[z,y_i]=\rank[x_{i-1},y_i],
$$
then we can replace $y_i$ with $z$.

It is clear that after finite steps of replacements, we get
a sequence which satisfies
\ref{item:condn}, \ref{item:nu diff rank} and \ref{item:no retrace}.
\end{proof}
We also note the following fact, whose proof is straightforward.

\begin{lemma}
\label{lem:deg ryx}
Let $\nu$ be an element of $\TTTTT(P)$ and 
$y_1$, $x_1$, \ldots, $y_t$, $x_t$ a sequence of elements in $P$
which satisfies \ref{item:condn} and \ref{item:nu diff rank} of Lemma \ref{lem:no retrace}.
Then
$$
\nu(x_0)\leq r(y_1,x_1,\ldots,y_t,x_t).
$$
\end{lemma}
Now we state

%%%%%%%%%%%%%%%%%%%%%%%%%%%
%%%%%%%%%%%%%%%%%%%%%%%%%%%
\iffalse
In order to prove this theorem, we recall the notion of 
\irseq\ \cite[p.\ 229]{miy}.

\begin{definition}\rm
Let 
$y_1$, $x_1$, $y_2$, $x_2$, \ldots, $y_t$, $x_t$ be a sequence of elements in $P$.
%If the following 3 conditions are satisfied, we say that
If the following 4 conditions are satisfied, we say that
$y_1$, $x_1$, $y_2$, $x_2$, \ldots, $y_t$, $x_t$ is an \irseq.
\begin{enumerate}
\item
$y_1$, $x_1$, $y_2$, $x_2$, \ldots, $y_t$, $x_t$ satisfies \condn.
\item
$r(y_1,x_1,\ldots,y_t,x_t)=\rmax$.
\item
If
$y'_1$, $x'_1$, $y'_2$, $x'_2$, \ldots, $y'_{t'}$, $x'_{t'}$ is a
sequence of elements in $P$ with \condn\  and
$r(y'_1,x'_1,\ldots,y'_{t'},x'_{t'})=\rmax$, 
then $t\leq t'$.
\item
For any $i$ with $1\leq i\leq t$,
$\rank[z,y_{i+1}]-\rank[z,y_i]
<\rank[x_i,y_{i+1}]-\rank[x_i,y_i]$
for any $z\in (x_i,y_{i+1}]\cap (x_i,y_i]$
and
$\rank[x_{i-1},z]-\rank[x_i,z]
<\rank[x_{i-1},y_{i}]-\rank[x_i,y_i]$
for any $z\in [x_{i-1},y_{i})\cap [x_i,y_i)$,
where we set $y_{t+1}=\infty$.
\end{enumerate}
\end{definition}
%
%
It is clear that there exists an \irseq.
\fi
%%%%%%%%%%%%%%%%%%%%%%%%%%%%%%%
%%%%%%%%%%%%%%%%%%%%%%%%%%%%%%%
%
%

\noindent
{\bf Proof of Theorem \ref{thm:non level main}}\
We first prove the ``only if'' part.
Take an exact sequence
$$
0\to R\to K_R(-a)\to C\to 0
$$
such that $\mu(C)=e(C)$ (since $R$ is not \gor, $C\neq 0$).
Since $R$ is not level, we see by \cite[Theorem 3.9]{miy} that
$\rmax>r$.
Therefore, by Fact \ref{fac:can gen deg} and Lemma \ref{lem:dim k c},
we see that
$$
\dim_k[C/\mmmm C]_1=\dim_k[K_R(-a)/\mmmm K_R(-a)]_1>0.
$$
Thus, we see by Fact \ref{fac:hilb c} and Lemma \ref{lem:mum em} that $s\geq 2$.
Further, since
$$
\dim_k[K_R(-a)/\mmmm K_R(-a)]_{s-2}\geq \dim_k[C/\mmmm C]_{s-2}=\dim_k[C/\mmmm C]_1>0
$$
by Corollary \ref{cor:sym dim}, we see by Fact \ref{fac:can gen deg} that
$$
\rmax\geq s-2-a=a+d-2-a=d-2.
$$

%%%%%%%%%%%%%%%%%%%%%%
%%%%%%%%%%%%%%%%%%%%%%
\iffalse
Take an \irseq\ 
$y_1$, $x_1$, \ldots, $y_t$, $x_t$
and
$\alpha_{ij}$ and $\beta_{ij}$
\fi
%%%%%%%%%%%%%%%%%%%%%%%%%%%
%%%%%%%%%%%%%%%%%%%%%%%%%%%
Let $\nu$ be a minimal element of $\TTTTT(P)$ with $\nu(x_0)=\rmax$.
Take a sequence $y_1$, $x_1$, \ldots, $y_t$, $x_t$ of elements in $P$ which
satisfies
\ref{item:condn}, \ref{item:nu diff rank} and \ref{item:no retrace}
of Lemma \ref{lem:no retrace}.
Further take
$\alpha_{ij}$ and $\beta_{ij}$
such that
$x_i=\alpha_{i0}\covered\alpha_{i1}\covered\cdots\covered\alpha_{iu_i}=y_{i+1}$
for $0\leq i\leq t$, where we set $y_{t+1}=\infty$ and $u_i=\rank[x_i,y_{i+1}]$ 
and
$x_i=\beta_{i0}\covered\beta_{i1}\covered\cdots\covered\beta_{iv_i}=y_{i}$
for $1\leq i\leq t$, where we set $v_i=\rank[x_i,y_i]$.
Then
$\alpha_{00}$, \ldots, $\alpha_{0,u_0}$, $\beta_{11}$, \ldots, $\beta_{1,v_1-1}$,
$\alpha_{10}$, \ldots, $\alpha_{1,u_1}$, $\beta_{21}$, \ldots, $\beta_{2,v_1-1}$,
\ldots,
$\alpha_{t-1,0}$, \ldots, $\alpha_{t-1,u_{t-1}}$, $\beta_{t,1}$, \ldots, $\beta_{t,v_{t}-1}$,
$\alpha_{t,0}$, \ldots, $\alpha_{t,u_t}$
are distinct elements of $P^+$,
since the sequence $y_1$, $x_1$, \ldots, $y_t$, $x_t$ 
satisfies \ref{item:condn}, \ref{item:nu diff rank} and \ref{item:no retrace} of
Lemma \ref{lem:no retrace}.
Since
\begin{eqnarray*}
u_0+v_1+\cdots+v_t+u_t-2&\leq&
\#P-2=d-2\\
&\leq&\rmax=\nu(x_0)\\
&\leq& r(y_1,x_1,\ldots,y_t,x_t)=u_0-v_1+\cdots-v_t+u_t,
\end{eqnarray*}
by Lemma \ref{lem:deg ryx},
we see that $t=1$, $v_1=1$ and
$\rmax=u_0-1+u_1
=\#P-2=d-2$.
In particular,
$P^+=\{\alpha_{00}$, $\alpha_{01}$, \ldots, $\alpha_{0u_0}$, $\alpha_{10}$,  $\alpha_{11}$, 
\ldots, $\alpha_{1u_1}\}$
and
$x_0=\alpha_{00}\covered\alpha_{01}\covered\cdots\covered\alpha_{0u_0}$,
$\alpha_{10}\covered\alpha_{11}\covered\cdots\covered\alpha_{1u_1}=\infty$ and
$\alpha_{10}\covered\alpha_{0u_0}$.

Set
$$
n\define\max\{j\mid \alpha_{0j}<\alpha_{10}\},\ 
m'\define\min\{j\mid\alpha_{1j}>\alpha_{0u_0}\},\ 
$$
$$
n'\define u_1-m'\mbox{ and }m\define u_0-n.
$$
Set also
\begin{eqnarray*}
&&w_0=\alpha_{00}, w_1=\alpha_{01}, \ldots, w_n=\alpha_{0n},\\
&&w'_0=\alpha_{1m'}, w'_1=\alpha_{1,m'+1}, \ldots, w'_{n'}=\alpha_{1u_1},\\
&&z_1=\alpha_{0,n+1}, z_2=\alpha_{0,n+2}, \ldots, z_m=\alpha_{0,u_0},\\
&&z'_1=\alpha_{10}, z'_2=\alpha_{11}, \ldots, z'_{m'}=\alpha_{1,m'-1}.
\end{eqnarray*}
Then \ref{item:whole set} and \ref{item:least cover rel} are satisfied.
Further, 
since $w_n\covered z'_1\covered z_m\covered w'_0$,
$z'_{m'}\covered w'_0$ and $w_n\covered z_1$,
we see that $z_1\neq z_m$ and $z'_1\neq z'_{m'}$, i.e., 
$m$, $m'\geq 2$, by the definition of covering relation.

Next we show that 
$$
\rank[x_0,x]+\rank[x,\infty]=r
$$
for any $x\in P^+$.
Assume the contrary.
Then by \cite[Lemma 4.1]{miy}, we see that
$\dim_k[K_R(-a)/\mmmm K_R(-a)]_0\geq 2$.
Therefore, we see by Corollary \ref{cor:sym dim} that
$\dim[K_R(-a)/\mmmm K_R(-a)]_{s-1}>0$.
It follows that 
$$
\rmax\geq s-1-a=a+d-1-a=d-1
$$
by Fact \ref{fac:can gen deg},
contradicts to the fact that $\rmax=d-2$.

Since
$\rank[x_0,z_m]=u_0=n+m$,
$\rank[z_m,\infty]=n'+1$,
$\rank[x_0,z'_1]=n+1$
and $\rank[z'_1,\infty]=u_1=m'+n'$,
we see that $n+m+n'+1=n+1+n'+m'=r$.
In particular, $m=m'$.
Further, we see that
%$n+m+n'+1=r$, 
$d=u_0+u_1+1=n+2m+n'+1$ and $s=d-r=m$.
Moreover, since 
$n+m+n'+1=r<\rmax=d-2=n+2m+n'-1$, 
we see that $s=m\geq 3$.

\medskip

Consider the possible covering relations in $P^+$ other than listed in
\ref{item:least cover rel}.
By the definition of covering relation,
$z'_i\not\covered z_j$ if $(i,j)\neq (1,m)$.
Therefore, we consider what kind of covering relations 
of the form $z_i\covered z'_j$ possible to exist.

If $i\geq j$ and $z_i\covered z'_j$, then
\begin{eqnarray*}
\rank[x_0,z_i]+\rank[z_i,z'_j]+\rank[z'_j,\infty]
&=&(n+i)+1+(m-j+n'+1)\\
&>&n+m+n'+1=r
\end{eqnarray*}
contradicting to the fact that $r=\rank P^+$.

Next assume that there are $i$ and $j$ such that
$j\geq i+2$, $(i,j)\neq(1,m)$ and $z_i\covered z'_j$.
Then, by the definition of covering relation, we see that
$z_1\not\covered z'_m$.
%%%%%%%%%%%%%%%%%%%%%%%%%%%
%%%%%%%%%%%%%%%%%%%%%%%%%%%
\iffalse
Thus, $z_m$, $z'_1$ is the only \irseq,
since we have shown above that for any \irseq\ $y_1$, $x_1$, \ldots, $y_t$, $x_t$,
$t=1$, $x_1\covered y_1$ and
$\rank[x_0,y_1]+\rank[x_1,\infty]=d-2$.
\fi
%%%%%%%%%%%%%%%%%%%%%%%%%%%
%%%%%%%%%%%%%%%%%%%%%%%%%%%
Let $\nu$ be a minimal element of $\TTTTT(P)$ such that $\nu(x_0)=\rmax=d-2$.
Take a sequence $y_1$, $x_1$, \ldots, $y_t$, $x_t$ of elements in $P$ 
which satisfies \ref{item:whole set}, \ref{item:nu diff rank} and \ref{item:no retrace}
of Lemma \ref{lem:no retrace}.
Then by the same argument as above, we see that $t=1$, $x_1\covered y_1$
and $\rank[x_0,y_1]+\rank[x_1,\infty]=d-1$.
Therefore, we see $y_1=z_m$ and $x_1=z'_1$ since $z_1\not\covered z'_m$.
Since
$\nu(x_0)-\nu(z_m)=\nu(x_0)-\nu(y_1)=\rank[x_0,y_1]=\rank[x_0,z_m]$
and
$\nu(z'_1)=\nu(x_1)-\nu(\infty)=\rank[x_1,\infty]=\rank[z'_1,\infty]$,
we see that
$$
\nu(x)=
\left\{
\begin{array}{ll}
\rank[x,\infty]&\quad\mbox{if $x\in [z'_1,\infty]\setminus\{z_m\}$,}\\
d-2-\rank[x_0,x]&\quad\mbox{otherwise.}
\end{array}
\right.
$$
In particular, there is exactly one minimal element $\nu$ of 
$\TTTTT(P)$ with $\nu(x_0)=d-2$.
Therefore,
$$
\dim_k[K_R(-a)/\mmmm K_R(-a)]_{s-2}=
\dim_k[K_R/\mmmm K_R]_{d-2}=1.
$$
On the other hand, if we define
$\nu_1$, $\nu_2\colon P^+\to\NNN$ by
\begin{eqnarray*}
\nu_1(x)&=&
\left\{
\begin{array}{ll}
\rank[x,\infty]&\quad\mbox{if $x\in [z'_1,\infty]\setminus\{z_m\}$,}\\
\rank[x,\infty]+1&\quad\mbox{otherwise,}
\end{array}
\right.\\
\nu_2(x)&=&
\left\{
\begin{array}{ll}
\rank[x,\infty]&\quad\mbox{if $x\in [z_i,\infty]\setminus\{z'_j\}$,}\\
\rank[x,\infty]+1&\quad\mbox{otherwise,}
\end{array}
\right.\\
\end{eqnarray*}
then $\nu_1$ and $\nu_2$ are minimal elements of $\TTTTT(P)$ by 
\cite[Lemma 3.2]{miy} and $\nu_1\neq\nu_2$ since $\nu_1(z_m)\neq\nu_2(z_m)$.
Further $\nu_1(x_0)=\nu_2(x_0)=r+1$.
Thus, we see that 
$$
\dim_k[K_R(-a)/\mmmm K_R(-a)]_1\geq 2.
$$
This contradicts to Corollary \ref{cor:sym dim}.
Therefore, we see that there is no covering relation $z_i\covered z'_j$ with
$j\geq i+2$ and $(i,j)\neq (1,m)$.

Summing up, possible covering relations of the form $z_i\covered z'_j$
are
\begin{itemize}
\item
$z_1\covered z'_m$
and
\item
$z_i\covered z'_{i+1}$ with $1\leq i\leq m-1$.
\end{itemize}
If $z_1\covered z'_m$, then there is no other covering relations of the form
$z_i\covered z'_j$ by the definition of covering relation.
Therefore, we see \ref{item:other cover rel} of Theorem \ref{thm:non level main}
is satisfied by
$z_1$, \ldots, $z_m$,
$z'_1$, \ldots, $z'_m$,
$w_0$, \ldots, $w_n$,
$w'_0$, \ldots, $w'_{n'}$.
This proves the ``only if'' part of Theorem \ref{thm:non level main}.

\medskip

Next we prove the ``if'' part of Theorem \ref{thm:non level main}.
We first assume that there are 
$z_1$, \ldots, $z_m$,
$z'_1$, \ldots, $z'_m$,
$w_0$, \ldots, $w_n$,
$w'_0$, \ldots, $w'_{n'}\in P^+$
which satisfy \ref{item:whole set},
\ref{item:least cover rel} and \ref{item:pure cover}.
First note that for any $\alpha_1$, $\alpha_2\in P^+$ with $\alpha_1\leq\alpha_2$,
$$
\rank[x_0,\alpha_1]+\rank[\alpha_1,\alpha_2]+\rank[\alpha_2,\infty]<r
$$
if and only if $(\alpha_1,\alpha_2)=(z'_1,z_m)$.
In particular, $\rank[x_0,x]+\rank[x,\infty]=r$ for any $x\in P^+$

For $1\leq h\leq m-1$, set
$\nu_h\colon P^+\to\NNN$ by
$$
\nu_h(x)=
\left\{
\begin{array}{ll}
\rank[x,\infty]&\quad\mbox{if $x\in [z'_1,\infty]\setminus\{z_m\}$,}\\
\rank[x,\infty]+m-1-h&\quad\mbox{otherwise.}
\end{array}
\right.\\
$$
Then by \cite[Lemma 3.2]{miy}, we see that $\nu_h$ is a minimal element of $\TTTTT(P)$,
 since $\nu_h(x_0)-\nu_h(z_m)=\rank[x_0,z_m]$ and $\nu_h(z'_1)=\rank[z'_1,\infty]$.
%$\rank[x_0,x]+\rank[x,\infty]=r$ for any $x\in P$.
Further,
$\nu_h(z'_1)-\nu_h(z_m)=h$ for $1\leq h\leq m-1$.

Let $\nu$ be an arbitrary minimal element of $\TTTTT(P)$.
Then by \cite[Lemma 3.3]{miy}, we see that there is a sequence
$y_1$, $x_1$, \ldots, $y_t$, $x_t$ with \condn\ such that
$$
\nu(x_{i-1})-\nu(y_i)=\rank[x_{i-1},y_i]
$$
for $1\leq i\leq t+1$, where we set $y_{t+1}=\infty$.
Take
$y_1$, $x_1$, \ldots, $y_t$, $x_t$ as short as possible.
If $t=0$, then $\nu(x_0)=r$. 
Thus, $\nu(x)=\rank[x,\infty]$ for any $x\in P$,
since $\rank[x_0,x]+\rank[x,\infty]=r$.
Therefore, $\nu=\nu_{m-1}$.
Suppose $t>0$.
Then $\nu(x_0)-\nu(y_2)>\rank[x_0,y_2]$
since we took 
$y_1$, $x_1$, \ldots, $y_t$, $x_t$ as short as possible.
Thus,
\begin{eqnarray*}
&&\rank[x_0,y_1]-\rank[x_1,y_1]+\rank[x_1,y_2]\\
&=&\nu(x_0)-\nu(y_1)-\rank[x_1,y_1]+\nu(x_1)-\nu(y_2)\\
&\geq&\nu(x_0)-\nu(y_1)-(\nu(x_1)-\nu(y_1))+\nu(x_1)-\nu(y_2)\\
&=&\nu(x_0)-\nu(y_2)\\
&>&\rank[x_0,y_2]
\end{eqnarray*}
and therefore,
%In particular,
\begin{eqnarray*}
&&\rank[x_0,x_1]+\rank[x_1,y_1]+\rank[y_1,\infty]\\
&=&2r-\rank[x_1,\infty]+\rank[x_1,y_1]-\rank[x_0,y_1]\\
&\leq&2r-\rank[x_0,y_1]+\rank[x_1,y_1]-\rank[x_1,y_2]-\rank[y_2,\infty]\\
&<&2r-\rank[x_0,y_2]-\rank[y_2,\infty]\\
&=&r.
\end{eqnarray*}
Thus,  $x_1=z'_1$ and $y_1=z_m$.
We see that $x_t=z'_1$ and $y_t=z_m$ by the same way.
Therefore $t=1$ since $y_1\not\geq x_i$ for any $i\geq 2$
by the definition of \condn.
Further, we see that $\nu=\nu_h$, where $h=\nu(z'_1)-\nu(z_m)$.
Thus, we see that there are exactly $m-1$ minimal elements
$\nu_1$, \ldots, $\nu_{m-1}$ of $\TTTTT(P)$.

Let $\varphi\colon R\to K_R(-a)$ be the $R$-homomorphism such  that 
$\varphi(1)=T^{\nu_{m-1}}$.
Since
\begin{eqnarray*}
K_R(-a)&=&\bigoplus_{\nu\in\TTTTT(P)}kT^\nu\\
&=&\left(\bigoplus_{h=1}^{m-2}\left(\bigoplus_{\nu\in\TTTTT(P)\atop\nu(z'_1)-\nu(z_m)=h}kT^\nu
\right)\right)
\oplus
\left(\bigoplus_{\nu\in\TTTTT(P)\atop\nu(z'_1)-\nu(z_m)\geq m-1}kT^\nu
\right)
\end{eqnarray*}
as a $k$-vector space and
$$
\image\varphi=
\bigoplus_{\nu\in\TTTTT(P)\atop\nu(z'_1)-\nu(z_m)\geq m-1}kT^\nu,
$$
we see that 
$$
\dim_k[\cok\varphi]_i=
\dim_k\left[
\bigoplus_{h=1}^{m-2}\left(\bigoplus_{\nu\in\TTTTT(P)\atop\nu(z'_1)-\nu(z_m)=h}kT^\nu
\right)\right]_i
$$
for any $i\in\ZZZ$.
Further, we see that $\mu(\cok\varphi)=m-2$, 
%since $\mu(K_R)=m-1$. 
since $K_R$ is generated by $T^{\nu_1}$, \ldots, $T^{\nu_{m-1}}$.
%and $T^{\nu_{m-1}}$ is a generator of $K_R$ by the above argument.
Set 
$$
R'=\bigoplus_{\nu\in\overline\TTTTT(P)\atop\nu(z'_1)=\nu(z_m)}kT^\nu.
$$
Then $R'$ is a $k$-subalgebra of $R$.
Further, $R'$ is isomorphic to a polynomial ring over $k$
with $d-1$ variables, since $R'$ is isomorphic to the Hibi ring over $k$
on a chain of length $d-2$.
Moreover, we see that
$
\bigoplus_{\nu\in\TTTTT(P)\atop\nu(z'_1)-\nu(z_m)=h}kT^\nu
$
is a rank 1 free module over $R'$ generated by $T^{\nu_h}$ for $1\leq h\leq m-2$.
Since multiplicity is defined by the Hilbert series, we see that
$
e(\cok\varphi)=m-2
$.
Therefore, we see that $R$ is \agor.

Next, we assume that there are
$z_1$, \ldots, $z_m$,
$z'_1$, \ldots, $z'_m$,
$w_0$, \ldots, $w_n$,
$w'_0$, \ldots, $w'_{n'}\in P^+$
which satisfy \ref{item:whole set}, \ref{item:least cover rel} and \ref{item:sym cover}.
For $1\leq h\leq m-1$, set
$\nu_h\colon P^+\to\NNN$ by
$$
\nu_h(x)=
\left\{
\begin{array}{ll}
\rank[x,\infty]&\quad\mbox{if $x\in [z'_1,\infty]\setminus\{z_m\}$,}\\
\rank[x,\infty]+m-1-h&\quad\mbox{otherwise}
\end{array}
\right.\\
$$
and for $1\leq h\leq m-2$, set
$\nu'_h\colon P^+\to\NNN$ by
$$
\nu'_h(x)=
\left\{
\begin{array}{ll}
\rank[x,\infty]&\quad\mbox{if $x\in [z_1,\infty]\setminus\{z'_m\}$,}\\
\rank[x,\infty]+m-1-h&\quad\mbox{otherwise.}
\end{array}
\right.\\
$$
Then we see by the same argument as above, that there are exactly $2m-3$
minimal elements
$\nu_1$, \ldots, $\nu_{m-1}$, $\nu'_1$, \ldots, $\nu'_{m-2}$ 
of $\TTTTT(P)$ and
$$
e(\cok\varphi)=\mu(\cok\varphi)=2m-4,
$$
where $\varphi$ is the $R$-homomorphism $R\to K_R(-a)$
such that $\varphi(1)=T^{\nu_{m-1}}$.
\qed

\begin{remark}\rm
$Q_m$ in \cite[Example 5.7]{hig} is the poset which satisfies 
\ref{item:condn}, \ref{item:least cover rel} and \ref{item:sym cover}
with $n=n'=0$
and $P_m$ in \cite[Theorem 5.3]{hig} is the poset which satisfies
\ref{item:condn}, \ref{item:least cover rel} and \ref{item:pure cover} with $n=n'=p=0$.
%%%%%%%%%%%%%%%%%
\iffalse
Further we see by the proof of Theorem \ref{thm:non level main} that
$$
[[C]]=\frac{2(\lambda+\lambda^2+\cdots+\lambda^{m-2})}{(1-\lambda)^{d-1}}
$$
if $P$ satisfies \ref{item:least cover rel} and \ref{item:sym cover}
and
$$
[[C]]=\frac{\lambda+\lambda^2+\cdots+\lambda^{m-2}}{(1-\lambda)^{d-1}}
$$
if $P$ satisfies \ref{item:least cover rel} and \ref{item:pure cover}.
Therefore, $h_{m-1}=h_1+2$ in the former case and $h_{m-1}=h_1+1$
in the latter case.
\fi
%%%%%%%%%%%%%%%%%%%%%%
\end{remark}

\section{\Agor\ property of ladder determinantal rings defined by 2-minors}
\label{sec:ladder}

In this section, we state a criterion of a ladder determinantal ring defined
by 2-minors to be a non-\gor\ \agor\ ring.
Note that a (ladder) determinantal ring defined by 2-minors is a Hibi ring.

We use the notation of \cite[\S 1]{con}.
Let $m$ and $n$ be integers with $2\leq m\leq n$ and $X$ an $m\times n$ matrix
of indeterminates.
Let $P_0$ be a poset which is a disjoint union of two chains
$\alpha_1\covered\cdots\covered\alpha_{m-1}$ and
$\beta_1\covered\cdots\covered\beta_{n-1}$.
Then by the correspondence
$$
(i,j)\leftrightarrow\{\alpha_1,\ldots,\alpha_{m-i}\}\cup\{\beta_1,\ldots,\beta_{j-1}\},
$$
we see the set of poset ideals (including the empty set) $J(P_0)$ of
$P_0$  ordered by inclusion is isomorphic to $X$ ordered by $\preceq$
of \cite[p.\ 121]{con}.
Let $P_1$ be a poset whose base set is $P_0$ and adding a covering relation 
$\alpha_s\covered\beta_t$.
Then a poset ideal of $P_1$ is a poset ideal of $P_0$, i.e.,
$J(P_1)\subset J(P_0)$.
Further, it is easily verified that by the above correspondence, $J(P_1)$
corresponds to
$$
\{(i,j)\mid i\leq m-s \mbox{ or } j\leq t \}
=X\setminus\{(i,j)\mid i>m-s, j>t\},
$$
i.e., $J(P_1)$ corresponds to a ladder with inside upper corner $(m-s,t)$
in the terminology of \cite[p.\ 122]{con}.

By repeating this argument, one sees that any ladder satisfying (a) and (b)
of \cite[p.\ 121]{con} can be obtained by adding several covering
relations of form
$\alpha_s\covered\beta_t$ or $\alpha_{s'}\covers\beta_{t'}$ to $P_0$.
Therefore, by the results of previous sections, we see the following fact.

\begin{thm}
\label{thm:ladder}
Let $m$, $n$ be integers with $2\leq m\leq n$,
$X$ an $m\times n$ matrix of indeterminates,
$Y$ a ladder of $X$ which satisfies (a) and (b) of \cite[p.\ 121]{con}.
Assume that all the indeterminates of $Y$ are involved in some 2-minors of $Y$.
Then $R_2(Y)$ is non-\gor\ and \agor\ if and only if one of the following
conditions is satisfied.
\begin{enumerate}
\item
$m=2$, $n\geq 3$ and $Y=X$.
\item
$m=n\geq 4$,
inside lower corner of $Y$ is $(2,2)$ and inside upper corner is 
just $(m-1,m-1)$ or lie on the line of equation $x+y=m+1$.
(We allow one sided ladder, i.e., there may be no inside upper corner.) 
\item
$m=n\geq 4$,
inside upper corner of $Y$ is $(m-1,m-1)$ and inside lower corner is 
just $(2,2)$ or lie on the line of equation $x+y=m+1$.
(We allow one sided ladder, i.e., there may be no inside lower corner.) 
\end{enumerate}
\end{thm}

\begin{remark}\rm
Suppose $t\geq 2$.
Taniguchi \cite{tan} showed that $R_t(X)$ is non-\gor\ and \agor\
if and only if $m=2$, $n\geq 3$ and $t=2$.
\end{remark}


\begin{thebibliography}{DEP2}
%
%
\bibitem[BF]{bf}
Barucci, Valentina, and Ralf Fr\"oberg.
"One-dimensional almost Gorenstein rings."
Journal of Algebra 188.2 (1997): 418-442.


\bibitem[Bas]{bas}
Bass, Hyman.
"On the ubiquity of Gorenstein rings."
Mathematische Zeitschrift 82.1 (1963): 8-28.


\bibitem[Con]{con}
Conca, Aldo.
 "Ladder determinantal rings."
 Journal of Pure and Applied Algebra 98.2 (1995): 119-134.


\bibitem[GMP]{gmp}
Goto, Shiro, Naoyuki Matsuoka, and Tran Thi Phuong.
"Almost Gorenstein rings."
Journal of Algebra 379 (2013): 355-381.


\bibitem[GTT]{gtt}
Goto, Shiro, Ryo Takahashi, and Naoki Taniguchi.
"Almost Gorenstein rings-towards a theory of higher dimension."
Journal of Pure and Applied Algebra 219.7 (2015): 2666-2712.


\bibitem[Hib]{hib}
Hibi, Takayuki.
"Distributive lattices, affine semigroup rings and algebras with straightening laws."
Commutative algebra and combinatorics 11 (1987): 93-109.


\bibitem[Hig]{hig}
Higashitani, Akihiro.
"Almost Gorenstein homogeneous rings and their h-vectors."
Journal of Algebra 456 (2016): 190-206.


\bibitem[Hoc]{hoc}
Hochster, Melvin.
"Rings of invariants of tori, Cohen-Macaulay rings generated by monomials, and polytopes."
Annals of Mathematics (1972): 318-337.


\bibitem[Miy]{miy}
Miyazaki, Mitsuhiro.
"On the generators of the canonical module of a Hibi ring: A criterion of level property and the degrees of generators."
Journal of Algebra 480 (2017): 215-236.

\bibitem[Sta]{sta2}
Stanley, Richard P.
"Hilbert functions of graded algebras."
Advances in Mathematics 28.1 (1978): 57-83.

\bibitem[Tan]{tan}
Taniguchi, Naoki.
"On the almost Gorenstein property of determinantal rings."
arXiv:1701.06690v2

%
%
%
%
%
%
\end{thebibliography}
\end{document}